\documentclass[notitlepage, onecolumn]{article}

\usepackage{graphicx}
\usepackage{amsmath, amsthm}
\usepackage{amssymb}
\usepackage[dvips]{epsfig}
\usepackage[noadjust]{cite}
\usepackage{setspace}
\usepackage[hmarginratio=1:1,top=32mm,columnsep=20pt]{geometry}
\usepackage{hyperref}

\newtheorem{theorem}{Theorem}
\newtheorem{remark}{Remark}

\newtheorem{assumption}{Assumption}
\newtheorem{proposition}{Proposition}
\newtheorem{definition}{Definition}
\newtheorem{corollary}{Corollary}

\newtheorem{objective}{Objective}

\newcommand{\mbf}{\mathbf}
\newcommand{\be}{\begin{equation}}
\newcommand{\ee}{\end{equation}}
\newcommand{\lbv}{\left\bracevert}
\newcommand{\rbv}{\right\bracevert}

\date{}

\title{On the Synchronization of Second-Order
Nonlinear Systems with Communication Constraints}
\author{A. Abdessameud{}, I. G. Polushin, and A. Tayebi
\thanks{This work was supported by the Natural Sciences and Engineering Research Council of Canada (NSERC).
Corresponding author A. Abdessameud. The authors are with the Department of Electrical and Computer Engineering, University of Western Ontario, London, Ontario, Canada. The third author is also with the Department of Electrical Engineering, Lakehead University, Thunder Bay, Ontario, Canada.
        {\tt\small aabdess@uwo.ca,  ipolushi@uwo.ca, tayebi@ieee.org}. }%
}

\begin{document}

\maketitle
\begin{abstract}
This paper studies the synchronization problem of second-order nonlinear multi-agent systems with intermittent communication in the presence of  irregular communication delays and possible information loss. The control objective is to steer all systems' positions to a common position with a prescribed desired velocity available to only some leaders.  Based on the small-gain framework, we propose a synchronization scheme relying on an intermittent information exchange protocol in the presence of time delays and possible packet dropout. We show that our control objectives are achieved with a simple selection of the control gains provided that the directed graph, describing the interconnection between all systems (or agents), contains a spanning tree. The example of Euler-Lagrange systems is considered to illustrate the application and effectiveness of the proposed approach.
\end{abstract}


\section{Introduction}
%
        Motion coordination of nonlinear multi-agent systems has received an increased interest in the control community due to the potential applications involving groups of robotic systems and autonomous vehicles in general \cite{qu2009cooperative, Ren:Cao:book}. Multi-agent systems control can be formulated as synchronization or consensus problems, where the goal is to drive the networked systems (or agents) to a common state using local information exchange. Other related problems include flocking, swarming, and formation control of mechanical systems. Built around the solutions of the consensus problem of linear multi-agent systems, several coordinated control schemes have been recently developed for second-order nonlinear dynamics, which can describe various mechanical systems, with a particular interest to leaderless synchronization problems \cite{abdess:TAC:2009, Ren:Lagrangian, Wang:flocking:2013, liu2013consensus}, cooperative tracking with full access to the reference trajectory \cite{Spong:Chopra:2007, chung:2009, abdess:TAC:2009}, leader-follower with single leader \cite{su2011adaptive, Mei:Ren:2011, Chen:Lewis:2011, meng2013robust, Meng:Dim:Joh:2014:ieeetro} or multiple leaders \cite{dimos:2009, Mei:Ren:2012, mei2013distributed}, to name only a few. Algebraic graph theory, matrix theory, and Lyapunov direct method have been shown useful tools to address various problems related to the systems dynamics, such as uncertainties, and the interconnection topology between the team members.
              %

%
            In addition, various recent papers address the synchronization problem of nonlinear systems by taking into account delays in the information transfer between agents, which is generally performed using communication channels. In \cite{Spong:Chopra:2007} and \cite{chopra2006passivity}, it has been shown that output synchronization of nonlinear passive systems is robust to constant communication delays if the interconnection graph is directed, balanced and strongly connected. A similar property was shown in \cite{chung:2009} under unbalanced directed graphs using the contraction theorem. In \cite{munz:2011}, a delay-robust control scheme is proposed for relative-degree two nonlinear systems with nonlinear interconnections. With the same assumption on the delays, adaptive synchronization schemes have been proposed in \cite{Nuno11, Wang:2013} for networked robotic systems under a directed graph. In addition to constant delays, a virtual systems approach has been suggested in \cite{abdess:IFAC:2011, aabdess:tay:book} to account for input saturations and to remove the requirements of velocity measurements. Control schemes that consider time-varying communication delays have also been proposed for the attitude synchronization of rigid body systems \cite{Erdong:2008, Abdess:attitude:TAC:2012}, formation control of unmanned aerial vehicles \cite{abdess:VTOL:2011}, and consensus of networked Lagrangian systems \cite{Nuno13}, yet in the case of undirected interconnection graphs. More recently, a small-gain framework is proposed in \cite{Abdessameud:Polushin:Tayebi:2013:ieeetac} for the synchronization of a class of second-order nonlinear systems in the presence of unknown irregular time-varying communication delays under general directed interconnection topologies.
             %

%
            One important problem when dealing with second-order nonlinear systems in the presence of communication delays is to achieve position synchronization, {\it i.e.,} all positions converge to a common value, with some non-zero final velocity. In fact, in most of the above mentioned synchronization laws with communication delays, a static leader or no leader are assumed and position synchronization is achieved with zero final velocity. The only cases where the final velocities match a non-zero value assume a full access to a reference trajectory or to a leader's states (position and velocity). By full access, it is meant that this information is available to all agents without delays. The main challenge in this case resides in the fact that imposing a non-zero final velocity ultimately requires some information on the delays to achieve position synchronization. In fact, a possible solution to this problem might be to explicitly incorporate the delays in the control algorithms as suggested in \cite{Munz:CDC, Zhu:Cheng:2010} for linear second-order multi-agents. This, however, comes with the assumptions of full access to the desired velocity and the communication delays are exactly known.
             %

%
            Another issue that can be observed in all the aforementioned results is the assumption that information is transmitted continuously between agents. In fact, it is not clear if these results still apply in situations where agents are allowed to communicate with their neighbors only during some disconnected intervals (or at some instants) of time. This can be induced by environmental constraints, such as communication obstacles, temporary sensor/communication-link failure, or imposed to the communication process to save energy/communication costs in mobile agents.
            For linear first-order multi-agent systems, the authors in \cite{sun2009consensus} have proposed a consensus algorithm based on the output of a zero-order-hold system, which is updated at instants when the information is received and admits as input the relative positions of interacting agents. In the presence of sufficiently small constant communication delays and bounded packet dropout, the proposed discontinuous algorithm in \cite{sun2009consensus} achieves consensus provided that self-delays are implemented and the non-zero update period of the zero-order-hold system is small.
            A similar approach has been applied for double integrators in \cite{gao2010asynchronous, gao2010consensus}, where asynchronous and synchronous updates of the zero-order-hold systems have been addressed, respectively, without communication delays.
            Here, synchrony means that all agents receive information at the same instants.
            In \cite{wen:duan:2012}, a switching algorithm has been proposed for second-order multi-agents in cases where communication between agents is lost during small intervals of time, yet without communication delays.
            The latter result has been extended to multi-agent systems with general linear dynamics \cite{wen:ren:2013} and globally Lipschitz nonlinear dynamics \cite{wen:l2:2012}, where it has been shown that consensus can be achieved under some conditions on the communication rates and interaction topology. 
              %

%
            In this paper, we consider the synchronization problem of a class of second-order nonlinear systems with intermittent communication in the presence of communication delays and possible packet loss. Here, it is required that all systems achieve position synchronization with some non-zero desired velocity available to only some systems in the group acting as leaders. Based on the small-gain approach, we propose a distributed control algorithm that allows agents to communicate with their neighbors only at some irregular discrete time-intervals and achieve our control objective. A discrete-time consensus algorithm is also used to handle the partial access to the desired velocity. In the case where no desired velocity is assigned to the team, the proposed synchronization algorithm achieves position synchronization with some velocity agreed upon by all agents. In both cases, it is proved that, under some sufficient conditions, synchronization is achieved in the presence of unknown irregular communication delays and packet loss provided that the interconnection topology between agents is described by a directed graph that contains a spanning tree. The derived conditions impose a maximum allowable interval of time during which a particular agent does not receive information from some or all of its neighbors. This interval, however, can be specified arbitrarily with a choice of the control gains. To illustrate the applicability of the proposed approach, we derive a solution to the above problems in the case of networked Lagrangian systems, and simulation results that show the effectiveness of the proposed approach are given.
%
%
\section{Background}\label{SecProbStat}
%

\subsection{Graph theory}\label{section_graph}
%
%
        Let $\mathcal{G}=(\mathcal{N},\mathcal{E})$ be a directed graph, with a set of nodes (or vertices) $\mathcal{N}$, and a set of ordered edges (pairs of nodes) $\mathcal{E}\subseteq\mathcal{N}\times\mathcal{N}$. An edge $(j, i)\in \mathcal{E}$ is represented by a directed link (arc) leaving node $j$ and directed toward node $i$. A directed graph $\mathcal{G}$ is said to contain a spanning tree if there exists at least one node that has a ``directed path'' to all the other nodes in the graph; by a directed path (of length $q$) from $j$ to $i$ is meant a sequence of edges in a directed graph of the form $(j,l_1),(l_1,l_2), \ldots, (l_{q-1}, l_q)$, with $l_q=i$, where for $q>1$ the nodes $j, l_1, \ldots, l_{q-1}\in\mathcal{N}$ are distinct. Node $r$ is called a root of $\mathcal{G}$ if it is the root of a directed spanning tree of $\mathcal{G}$; in this case, $\mathcal{G}$ is said to be rooted at $r$.
%

%
        Given two graphs $\mathcal{G}_1=(\mathcal{N},\mathcal{E}_1)$, $\mathcal{G}_2=(\mathcal{N},\mathcal{E}_2)$ with the same vertex set $\mathcal{N}$, their composition $\mathcal{G}_3:=(\mathcal{N},\mathcal{E}_3)=\mathcal{G}_1\circ \mathcal{G}_2$ is the graph with the same vertex set $\mathcal{N}$ where $(j,i)\in \mathcal{E}_3$ if and only if $(j,l)\in \mathcal{E}_2$ and $(l,i)\in \mathcal{E}_1$ for some $l\in \mathcal{N}$. Composition of any finite number of graphs is defined by induction. In the case where $\mathcal{G}_1$ and $\mathcal{G}_2$ contain self-links at all nodes, the edges of $\mathcal{G}_1$ and $\mathcal{G}_2$ are also edges of $\mathcal{G}_3$. In this case, the definition above also implies that $\mathcal{G}_3$ contains a path from $j$ to $i$ if and only if $\mathcal{G}_2$ contains a path from $j$ to $l$ and $\mathcal{G}_1$ contains a path from $l$ to $i$. A finite sequence of directed graphs $\mathcal{G}_1$, $\mathcal{G}_2$, \ldots, $\mathcal{G}_q$ with the same vertex set is jointly rooted if the composition $\mathcal{G}_q  \circ \mathcal{G}_{q-1}\circ \ldots \circ \mathcal{G}_1$ is rooted. An infinite sequence of graphs $\mathcal{G}_0$, $\mathcal{G}_1$, $\ldots$ is said to be repeatedly jointly rooted if there exists $k^*\in{\mathbb Z}_+$ such that for any $\sigma\in{\mathbb Z}_+$ the finite sequence $\mathcal{G}_\sigma$, $\mathcal{G}_{\sigma+1}, \ldots$, $\mathcal{G}_{\sigma+k^*}$ is jointly rooted. (See~\cite{Cao:etal:2008:1:SIAMJCO} for more details on graph composition).

        A weighted directed graph $\mathcal{G}_{w}$ consists of the triplet $(\mathcal{N},\mathcal{E}, \mathcal{A})$, where $\mathcal{N}$ and $\mathcal{E}$ are, respectively, the sets of nodes and edges defined as above, and $\mathcal{A}$ is the weighted adjacency matrix defined such that $a_{ii}\triangleq0$, $a_{ij}>0$ if  $(j,i)\in\mathcal{E}$, and  $a_{ij}=0$ if $(j,i)\notin \mathcal{E}$. Note that thus defined graph does not contain self-links at any node and will have the same properties as the unweighted graph with the same sets of nodes and edges. The Laplacian matrix $\mathbf{L}:=[l_{ij}]\in\mathbb{R}^{n\times n}$ of the weighted directed graph $\mathcal{G}_w$ is defined such that: $l_{ii}=\sum_{j=1}^n a_{ij}$, and $l_{ij}=-a_{ij}$ for $i\neq j$.

\subsection{Stability Notions and preliminary result}
%
        Consider an affine nonlinear system of the form
            \begin{equation}
            \label{affine001}
            \begin{array}{rcl}
            {\dot x}&=&f(x)+ \sum_{i=1}^p g_i(x)u_i,\\ 
            y_j & = & h_j(x), \quad ~j= 1,\ldots,q,
            \end{array}
            \end{equation}
        where  ${x}\in\mathbb{R}^N$, ${u}_i\in\mathbb{R}^{\tilde{m}_i}$ for $i\in\mathcal{N}_p:=\{1,\ldots,p\}$, ${y}_j\in\mathbb{R}^{\bar{m}_j}$ for $j\in\mathcal{N}_q:=\{1,\ldots,q\}$, and $f(\cdot )$, $g_i(\cdot )$, for $i\in\mathcal{N}_p$, and $h_j(\cdot )$, for $j\in\mathcal{N}_q$, are locally Lipschitz functions of the corresponding dimensions, $f(0)=0$, $h(0)=0$. We assume that for any initial condition $x(t_0)$ and any inputs $u_1(t)$, \ldots, $u_p(t)$ that are uniformly essentially bounded on $[t_0, t_1)$, the corresponding solution $x(t)$ is well defined  for all $t\in [t_0, t_1]$.
        \begin{definition}\cite{sontag:06:1}\label{def_iss}
           A system of the form \eqref{affine001} is said to be input-to-state stable (ISS) if there exist\footnote{ The definition of the class functions $\mathcal{K}$, $\mathcal{K}_{\infty}$, and $\mathcal{K}\mathcal{L}$ can be found in ~\cite{khalil:02}. Also, ${\bar {\mathcal K}}:={\mathcal K}\cup\{{\mathcal O}\}$, ${\bar {\mathcal K}}_{\infty}:={\mathcal K}_{\infty}\cup\{{\mathcal O}\}$, where ${\mathcal O}$ is zero function, ${\mathcal O}(s)\equiv 0$ for all $s\ge 0$.}
                 $\beta\in {\mathcal K}_{\infty}$ and $\gamma_{j}\in {\bar{\mathcal K}}$, $j\in\mathcal{N}_p$, such that the following inequalities hold along the trajectories of the system for any Lebesgue measurable uniformly essentially bounded inputs $u_j$, $j\in\mathcal{N}_p$:
%

%
            	\begin{itemize}
              	\item [i)] {\it uniform boundedness:} ~$\forall~t_0,~ t\in{\mathbb R},~ t\ge t_0$, we have \vspace{-0.2 cm}
              	\begin{equation}\begin{array}{c}\left|x(t)\right|\le \beta\left(|x(t_0)|\right)+\sum\limits_{j=1}^{p} \gamma_{j}\left(\sup\limits_{s\in\left[t_0, t\right)}|u_j(s)|\right),\end{array}\nonumber\end{equation}
%
%
              \item[ii)] {\it asymptotic gain:} 
              	\begin{equation}\begin{array}{c}
              \limsup\limits_{t\to +\infty} \left|x(t)\right|\le \sum\limits_{j=1}^{p}\gamma_{j}\left(\limsup\limits_{t\to +\infty} |u_j(t)|\right).\end{array}\nonumber\end{equation}
            	\end{itemize}
           \end{definition}
%
%
        In the above definition, $|\cdot|$ denotes the Euclidean norm of a vector and $\gamma_j\in\bar{\mathcal{K}}$, $j\in\mathcal{N}_p$, are called the ISS gains. It should be pointed out that for a system of the form \eqref{affine001}, the ISS implies the input-to-output stability  (IOS) \cite{sontag:06:1}, which means that there exist $\beta_i\in {\mathcal K} \mathcal{L}$ and $\gamma_{ij}\in {\bar{\mathcal K}}$, $i\in\mathcal{N}_q$, $j\in\mathcal{N}_p$, such that the inequality
        \begin{equation}\begin{array}{c}\left|y_i(t)\right|\le \beta_i\left(|x(t_0)|, t \right)+ \sum\limits_{j=1}^{p} \gamma_{ij}\left(\sup\limits_{s\in\left[t_0, t\right)}|u_j(s)|\right),\end{array}\nonumber\end{equation}
        holds for all $i\in\mathcal{N}_q$  and $\forall~t_0, ~t\in{\mathbb R},~ t\ge t_0$.
        In this case, the function $\gamma_{ij}\in {\bar {\mathcal K}}$, $i\in\mathcal{N}_q$ and $j\in\mathcal{N}_p$, is called the IOS gain from the input $u_j$ to the output $y_i$. In the subsequent analysis, we will mostly deal with the case where the IOS gains are linear functions of the form $\gamma_{ij}(s):=\gamma_{ij}^0\cdot s$,  where $\gamma_{ij}^0\ge 0$; in this case, we will simply say that the system has linear IOS gains $\gamma_{ij}^0\ge 0$.
%

%
        The convergence analysis in this paper is based on the following small-gain theorem.
         \begin{theorem} \label{theorem001a}
          Consider a system of the form \eqref{affine001}. Suppose the system is IOS with linear IOS gains $\gamma_{ij}^0\ge 0$. Suppose also that each input $u_j(\cdot)$, $j\in\mathcal{N}_p$, is a Lebesgue measurable function satisfying:
                            $u_j(t)\equiv 0$, for $t<0$, 
          and                     \begin{equation}
                            |u_j(t)|\le\sum\limits_{i\in\mathcal{N}_q} {\mu}_{ji}\cdot\sup\limits_{s\in\left[t-\vartheta_{ji}(t), t\right]} \left| y_i(s)\right| +|\delta_j(t)|,
                            \label{commconstraints0010}
                    \end{equation}
          for almost all $t\ge 0$, where $ {\mu}_{ji}\ge 0$,  all $\vartheta_{ji}(t)$ are Lebesgue measurable uniformly bounded nonnegative functions of time, and $\delta_j(t)$ is an uniformly essentially bounded signal. Let $\Gamma:=\Gamma^0\cdot {\mathcal M}\in{\mathbb R}^{q\times q}$, where $\Gamma^0:=\left\{ \gamma^0_{ij} \right\}$, ${\mathcal M}:=  \left\{{\mu}_{ji}\right\}$, $i\in\mathcal{N}_q$, $j\in\mathcal{N}_p$. If $\rho\left(\Gamma \right)<1$, where $\rho\left(\Gamma \right)$ is the spectral radius of the matrix $\Gamma$, then the trajectories of the system (\ref{affine001}) 
          are well defined for all $t\ge 0$ and such that all the outputs $y_i(t)$, $i\in\mathcal{N}_q$, and all the inputs  $u_j(\cdot)$, $j\in\mathcal{N}_p$, are uniformly bounded. If, in addition, $|\delta_j(t)|\to 0$ at $t\to +\infty$, $j\in\mathcal{N}_p$, then $\left|y_i(t)\right|\to 0$, $\left|u_j(t)\right|\to 0$ as $t\to +\infty$ for $i\in\mathcal{N}_q$ and $j\in\mathcal{N}_p$.
           	$\square$
        \end{theorem}
        Theorem~\ref{theorem001a} is a version of \cite[Theorem 1]{Abdessameud:Polushin:Tayebi:2013:ieeetac} and is also a special case of the result given in~\cite{polushin:dashkovskiy:takhmar:patel:13:automatica}; in particular, its proof follows the same lines as in the  proof in \cite[Theorem 1]{Abdessameud:Polushin:Tayebi:2013:ieeetac}, and hence, is omitted.

\section{Problem Formulation}
%
\subsection{System Model}
%
            Consider $n$ not necessarily identical second-order nonlinear systems (or agents) governed by  the following dynamics
                \begin{equation}\label{model}\begin{array}{lcl}
                \dot{p}_i(t) &=& v_i(t),\\
                \dot{v}_i(t) &=& F_i(p_i(t), v_i(t), u_i(t)),\end{array}
                \qquad i\in\mathcal{N},
                \end{equation}
            where  $p_i \in\mathbb{R}^m$ and $v_i$ are the position-like and velocity-like states, respectively, $u_i \in\mathbb{R}^m$ are the inputs,
            and $\mathcal{N}:=\{1,\ldots, n\}$. The functions $F_i$ are assumed to be locally Lipschitz with respect to their arguments. Note that equations \eqref{model} can be used to describe the full or part of the dynamics of various physical systems.
%

%
            The systems \eqref{model} are interconnected in the sense that some information can be transmitted between agents using communication channels. This interconnection is represented by a directed graph $\mathcal{G}=(\mathcal{N},\mathcal{E})$, where $\mathcal{N}$ is the set of all agents, and an edge $(j,i)\in\mathcal{E}$ indicates that the $i$-th agent {\it can} receive information from the $j$-th agent; in this case, we say that $j$ and $i$ are neighbors (even though the link between them is directed). While the interconnection graph $\mathcal{G}$ is fixed, the information exchange between agents is not continuous but discrete in time and is subject to communication constraints as described in the next subsection.

\subsection{Communication Process}\label{section_comm}
%
            In this paper, we consider the case where the communication between agents is intermittent and is subject to time-varying communication delays, information losses, and blackout intervals. Specifically, it is assumed that there exists a strictly increasing and unbounded sequence of time instants $t_k:=k T\in {\mathbb R}_+$, $k\in\mathbb{Z}_+= \{0,1,\ldots\}$, where $T>0$ is a fixed sampling period common for all agents, such that each agent is allowed to send its information to all or some of its neighbors at instants $t_k$, $k\in\mathbb{Z}_+$. In addition, for each pair $(j,i)\in\mathcal{E}$, suppose that there exist a sequence of communication delays $(\tau_k^{(j,i)})_{k\in\mathbb{Z}_+}$ that take values in $\{{\mathbb R}_+\cup{+\infty}\}$ such that the information sent by agent $j$ at instant $t_k$ can be available to agent $i$ starting from the instant $t_k+\tau_k^{(j,i)}$. In particular, it is possible that $\tau_k^{(j,i)}={+\infty}$ for some $k\in\mathbb{Z}_+$, which corresponds to a situation where agent $j$ has not sent information at instant $t_k$ to neighbor $i$ at all, or the corresponding information was never received possibly due to packet loss in the communication channel. The following assumption is imposed on the communication process between neighboring agents.
            \begin{assumption}\label{AssumptionDelay01}
                For each $(j,i)\in\mathcal{E}$, there exist numbers $k^*\in{\mathbb N}$, $h \ge 0$, and an infinite strictly increasing sequence ${\mathcal K}^{(j,i)}:= \{k^{(j,i)}_0,k^{(j,i)}_1,\ldots \}\subset \{0,1,\ldots \}$ satisfying
                   \begin{itemize}
                    \item [i)] $k^{(j,i)}_0\le k^*$, and $k^{(j,i)}_{l+1} - k^{(j,i)}_l\le k^*$, $l\in\{0,1,\ldots \}$,
                    \item[ii)] $\tau_k^{(j,i)}\le h$ for each $k\in{\mathcal K}^{(j,i)}$.
                    \end{itemize}
            \end{assumption}

            Assumption~\ref{AssumptionDelay01} essentially means that, for each pair $(j,i)\in\mathcal{E}$, and per any $k^*$ consecutive sampling instants, there exists at least one sampling instant at which agent $j$ has sent information to agent $i$, and this information has been successfully delivered with delay less than or equal to $h$. Note that $h$ is not an imposed upper bound of the delays; in particular, the possible case where  $\tau_\sigma^{(j,i)} > h$ for some $\sigma \notin {\mathcal K}^{(j,i)}$ is not excluded.
            Assumption~\ref{AssumptionDelay01} also implies that, for each pair $(j,i)\in\mathcal{E}$, the maximal interval between two consecutive instants when agent $i$ receives information from agent $j$ is less than or equal to
            \begin{equation}
            \label{h_star} h^*:=k^* T+ h.
            \end{equation}
            It is worth pointing out that $k^*$, $h$, and $h^*$ are considered common to all $(j,i)\in\mathcal{E}$ for simplicity, and can be seen as the maximum of the corresponding parameters defined for each pair $(j,i)\in\mathcal{E}$. Also, the above assumption does not require that all agents broadcast their information to some or all of their prescribed neighbors at the same instants of time.

\subsection{Problem statement}
%
%
            Consider multi-agents \eqref{model} interconnected according to $\mathcal{G}$ and the communication between agents satisfies Assumption~\ref{AssumptionDelay01}. Suppose also that a constant desired velocity $v_d \in\mathbb{R}^m$ is available for a subset ${\mathcal L}\subset {\mathcal N}$ of agents, called leaders. The rest of the systems (belonging to the complementary subset ${\mathcal F}:= {\mathcal N}\setminus {\mathcal L}$) are referred to as followers. Our goal is to design synchronization schemes for the nonlinear multi-agent system \eqref{model} such that the following objectives are attained.
%
            \begin{objective}\label{objective1} In the case $\mathcal{L}\ne \emptyset$, it is required that
                    $p_i(t) - p_j(t)\to 0$, $v_i(t) \to v_d$, as $t\to +\infty$ for all $i,j\in\mathcal{N}$.
            \end{objective}
            \begin{objective}\label{objective2} In the case $\mathcal{L}=\emptyset$, it is required that
                                   $p_i(t) - p_j(t)\to 0$ and $v_i(t) \to v_c$ as $t\to +\infty$ for all $i,j\in\mathcal{N}$ and for some final velocity $v_c\in\mathbb{R}^m$.
             \end{objective}
%
%
                To achieve the above objectives, we adopt an approach that takes its roots from the control of robotic systems \cite{slotine1987adaptive} and has been recently used to address various synchronization problems of
                mechanical systems (see, for instance, \cite{abdess:VTOL:2011, Nuno11, Chen:Lewis:2011, Wang:flocking:2013, Abdessameud:Polushin:Tayebi:2013:ieeetac}).
                To explain this approach, let $v_{r_i}(t)\in \mathcal{C}^1$ be a reference velocity for the $i$-th system, for $i\in\mathcal{N}$. Equations \eqref{model} can be rewritten in the following form,
                    \begin{eqnarray}
                        \label{model_reduced}\dot{p}_i(t) &=& v_{r_i}(t) + e_i(t),\\
                        \label{model_closed_error}\dot{e}_i(t) &=& F_i (p_i, e_i + v_{r_i}, u_i) - \dot{v}_{r_i},
                   \end{eqnarray}
                $i\in\mathcal{N}$, where $e_i:= v_i - v_{r_{i}}$ is the velocity tracking error. Equations \eqref{model_reduced} describe the dynamics of agents in a multi-agent system with $v_{r_i}(t)$ being the reference input, while  $e_i(t)$ is a perturbation term with dynamics described by \eqref{model_closed_error}.   The synchronization problem can now be solved using a two stages approach described as follows. In the first stage, the input $u_i(t)$ in \eqref{model_closed_error} is designed to guarantee the convergence of the error signals $e_i(t)$ to zero. In the second stage, appropriate algorithms for $v_{r_i}(t)$ are designed using the position-like states of the systems such that the trajectories of the dynamic systems \eqref{model_reduced} satisfy Objectives \ref{objective1} or \ref{objective2}. As mentioned in the Introduction, the problem of designing such algorithms for $v_{r_i}(t)$ in the presence of the communication constraints described in Section~\ref{section_comm} is yet unsolved, even in the case where no perturbation term exists, {\it i.e.}, $e_i\equiv 0$. In the presence of nonzero signals $e_i(t)$, the problem becomes more complicated since there typically exists some coupling between the signals $e_i(t)$ and $v_{r_i}(t)$ as they both depend on the states of \eqref{model_reduced}-\eqref{model_closed_error}. Note that, the first step, the design of the control law $u_i$ in \eqref{model_closed_error} that guarantees desirable properties of the error signal $e_i$, for a given $v_{r_i}$ and $\dot{v}_{r_i}$, can be achieved using various existing approaches to tracking control design for nonlinear systems.  This step is not addressed in this work; instead, the following assumption is made.
                \begin{assumption}\label{designcond}
                For each system in \eqref{model} and a given reference velocity signals $v_{r_i}(t)\in \mathcal{C}^1$ and $\dot{v}_{r_i}(t)$, there exists a static or dynamic tracking control law $u_i$ such that the following hold:\vspace{-0.1in}
                             \begin{itemize}
                             \item the error signal $e_i(t)$ is uniformly bounded;
                              \item if $v_{r_i}(t)$ and $\dot{v}_{r_i}(t)$ are globally uniformly bounded, then $e_i(t)\to 0$ as $t\to+\infty$.
                              \end{itemize}
                \end{assumption}
                Under Assumption~\ref{designcond}, the synchronization problem of the nonlinear multi-agent system \eqref{model} is reduced to the design of the reference velocities $v_{r_i}\in \mathcal{C}^1$,  $i\in\mathcal{N}$, such that $\dot{v}_{r_i}$ is well defined (available for feedback) and the trajectories of \eqref{model_reduced} satisfy Objectives \ref{objective1} or \ref{objective2}.

\section{Synchronization Scheme Design}
%
%
                In this section, we present a method for the design of the reference velocities $v_{r_i}(t)$, $i\in\mathcal{N}$, such that Objectives  \ref{objective1} and \ref{objective2} are achieved using intermittent communication between the agents in the presence of time-varying communication delays and information losses. For this, we let $\mbf{p}_j^i(k)$, for each $(j,i)\in\mathcal{E}$ and each $k\in{\mathbb Z}_+$, denote the information that can be transmitted from agent $j$ to agent $i$ at $t = kT$. Specifically, $\mbf{p}_j^i(k):=[p_j(kT), \hat{v}_{d_j}(k), k]$, where $p_j(kT)$ is the position-like state of the $j$-th system at $t=kT$, $\hat{v}_{d_j}(k)$  is a discrete-time\footnote{For simplicity, we use throughout the paper the notation $ x(k)$, $k\in\mathbb{Z}_+$, instead of $x(kT)$ for the discrete-time signals.} estimate of the desired velocity obtained by the $j$-th agent according to an algorithm described below, and $k$ is the time-stamp, {\it i.e.,} the sequence number at which information was sent. Also, for each pair $(j,i)\in\mathcal{E}$ and each time instant $t\geq 0$, we let $k_{ij}(t)$ be the largest integer number such that $\mbf{p}_j^i(k_{ij}(t))=[p_j(k_{ij}(t)T), \hat{v}_{d_j}(k_{ij}(t)), k_{ij}(t)]$ is the most recent information of agent $j$ that is already delivered to agent $i$ at $t$, {\it i.e.,}
              \begin{equation}\label{k_t}
                            k_{ij}(t):=\max\{k\in{\mathbb Z}_+:\; kT+\tau_k^{(j,i)}\le t\}.
              \end{equation}
              Note that the number $k_{ij}(t)$ can be determined by a simple comparison of the received time stamps.
%

%
              Now, for each $i\in\mathcal{N}$, the reference velocity $v_{r_i}(t)$ is designed in the following form,
                 \begin{equation}\label{v_reference}
                    v_{r_i}(t) = \eta_i(t) + {\bar v}_{d_i}(t),
                \end{equation}
              where $ {\bar v}_{d_i}$ is a sufficiently smooth estimate of the desired velocity available for the $i$-th agent, and $\eta_i$ is a synchronization term designed with the purpose of position synchronization between agents. The design of ${\bar v}_{d_i}$ and $\eta_i$ are addressed below in detail.

\subsection{Desired velocity estimation}
%
%
        In this subsection, we present a method for the design of ${\bar v}_{d_i}$ in \eqref{v_reference}. As explained above, each leader has direct access to the desired velocity $v_d \in\mathbb{R}^m$. The followers, on the other hand, do not have direct access to the desired velocity $v_d$; instead, they estimate it through the following discrete-time consensus algorithm that is updated at instants $\sigma T$, for $\sigma\in{\mathbb Z}_+$,
            \begin{eqnarray}\label{updatevk_leader}
                    \hat{v}_{d_i}(\sigma)&\equiv&  v_d\quad \mbox{ for }~i\in {\mathcal L},\\
                \label{updatevk_follower}
                \hat{v}_{d_i}(\sigma+1)&=&
                    \frac{1}{\left\bracevert{N_i(\sigma)}\right\bracevert}\sum\limits_{j\in N_i(\sigma)}\hat{v}_{d_{ij}}(\sigma) \quad\mbox{ for} ~i\in {\mathcal F},
                \end{eqnarray}
      where
      %

%
                \begin{equation}\label{updatevk_f2}
                    \hat{v}_{d_{ij}}(\sigma):=
                    \left\{
                    \begin{array}{ll}
                        \hat{v}_{d_j}(k_{ij}(\sigma T)) & \mbox{ if }~ j\neq i,  \\ 
                        \hat{v}_{d_i}(\sigma)& \mbox{ if }~ j = i. 
                    \end{array}
                    \right.
                \end{equation}

             In the above algorithm \eqref{updatevk_follower}-\eqref{updatevk_f2}, $N_i(\sigma):=\{i\}\cup N_{ij}(\sigma)$, with $$N_{ij}(\sigma):=\{j: (j,i)\in\mathcal{E},~ k_{ij}(\sigma T) > k_{ij}((\sigma-1)T)\},$$
             and $\left\bracevert{N_i(\sigma)}\right\bracevert$ denotes the number of elements in $N_i(\sigma)$. Recall from  \eqref{k_t} that the vector $\hat{v}_{d_j}(k_{ij}(\sigma T))$ is the most recent desired velocity estimate (obtained by agent $j$) that is already available to agent $i$ at instant $\sigma T$. As such, the set $N_{ij}(\sigma)$ denotes the set of the neighbors of the $i$-th follower such that the most recent data from these neighbors has been received during the interval $\left((\sigma-1)T, \sigma T\right]$. Therefore, the update law for the followers \eqref{updatevk_follower}-\eqref{updatevk_f2} is based on each agent's own estimate of the desired velocity as well as on the most recent velocity estimates (received from its neighbors) that have not been already used in the update law at earlier sampling instants. The leaders, on the other hand, do not update their estimates as can be seen from \eqref{updatevk_leader}. It should be noted that the set $N_i(\sigma)$ used in \eqref{updatevk_follower}-\eqref{updatevk_f2} can be obtained by simple comparison between the successfully received time stamps.
             %

%
           Using the discrete-time desired velocity estimates obtained by each agent, we let
               \begin{eqnarray}
            \label{hat_v_filter_leader} {\bar v}_{d_i} &\equiv& v_d \quad\mbox{for}~i\in\mathcal{L},\\
            \label{hat_v_filter}
                 \ddot{\bar v}_{d_i}&=&-k_i^d \dot{\bar v}_{d_i} - k_i^p \left({\bar v}_{d_i}-  \hat{v}_{d_i}\left(\lfloor t/T\rfloor \right)\right)
                 \;\mbox{for}~i\in\mathcal{F},
            \end{eqnarray}
            where $\lfloor a \rfloor$ denotes the integer part of $a\in{\mathbb R}_+$, and $k_i^p$, $k_i^d$ are strictly positive scalar gains. Note that $\bar{v}_{d_i}\in\mathcal{C}^1$. \\
%
             \begin{proposition}\label{prop21}
                  Consider the discrete-time consensus algorithm \eqref{updatevk_leader}-\eqref{updatevk_f2}. Suppose Assumption~\ref{AssumptionDelay01} holds. If $\mathcal{L}\neq\emptyset$ and the directed interconnection graph $\mathcal{G}$ is rooted at $r\in\mathcal{L}$, then $\hat{v}_{d_i}(\sigma)\to v_c = v_d$ as $\sigma\to+\infty$ for all $i\in\mathcal{N}$. Also, if  $\mathcal{L}=\emptyset$ and $\mathcal{G}$ is rooted, then
                  $\hat{v}_{d_i}(\sigma)\to v_c$ as $\sigma\to+\infty$ for all $i\in\mathcal{N}$, for some $v_c\in\mathbb{R}^m$. In both cases, ${\bar v}_{d_i}(t)$ and $\dot{\bar v}_{d_i}(t)$ are uniformly bounded and ${\bar v}_{d_i}(t) \to v_c$ as $t\to +\infty$, for all $i\in\mathcal{N}$.~$\square$
            \end{proposition}

\subsection{Design of the synchronization terms}
%
%
            In this subsection, the design of the synchronization term $\eta_i$ in \eqref{v_reference} is addressed. For this, let $\mathcal{A}=[a_{ij}]\in\mathcal{R}^{n\times n}$ be an arbitrary weighted adjacency matrix (Defined in Section~\ref{section_graph}) assigned to the graph $\mathcal{G}$; the resulting weighted directed graph is denoted by $\mathcal{G}_w = (\mathcal{N}, \mathcal{E}, \mathcal{A})$. Let $\mathcal{N}_i:=\{j\in \mathcal{N}:\; (j,i)\in\mathcal{E}\}$ denote the set of neighbors of agent $i$ in $\mathcal{G}_w$. Also, let $\kappa_i:=\sum_{j=1}^n a_{ij}$ and $\mathcal{N}^\sharp:=\{ i\in \mathcal{N}\colon\; \kappa_i \neq 0\}$ denote the subset of all agents that have at least one incoming link in $\mathcal{G}_w$.
  %

%
            Consider the following design of the synchronization term $\eta_i$ in \eqref{v_reference}
            \begin{equation}\label{filter}
                \left\{
                \begin{array}{ccl}
                            \dot{\eta}_i &=& - k_i^{\eta} \eta_i - \lambda_i (p_i - \psi_i), \\
                            \dot{\psi}_i &=& - \psi_i + {\bar v}_{d_i} + \frac{1}{\kappa_i}\sum_{j\in\mathcal{N}_i} a_{ij} p^{(i)}_{j}(t),
                \end{array}
                            \right.
                \end{equation}
            for $i\in \mathcal{N}^\sharp$,
                \begin{equation}\label{eta_zeros}
                    \dot{\eta}_i = \eta_i =  0\quad\mbox{for}~i\in\mathcal{N} \setminus \mathcal{N}^\sharp,
                \end{equation}
            where $k_i^{\eta}$, $\lambda_i$ are strictly positive scalar gains, $\bar{v}_{d_i}$ is defined in \eqref{hat_v_filter_leader}-\eqref{hat_v_filter}, and the vector $p^{(i)}_{j}(t)$, for all $j\in\mathcal{N}_i$, is an estimate of the current position of the $j$-th agent defined using the most recent information available to the $i$-th agent at $t$ as
            \begin{eqnarray}\label{qijversion2_1}
                    p^{(i)}_{j}(t)&:=& p_j(k_{ij}(t) T)+  \epsilon_{ij}(t),\\
                    \label{epsilon}\epsilon_{ij}(t)&:=& \hat{v}_{d_j} (k_{ij}(t))\cdot (t - k_{ij}(t) T).
                \end{eqnarray}
            with $k_{ij}(t)$ being defined in \eqref{k_t}.

            \begin{remark}
            Note that, due to the intermittent and delayed nature of the communication process, we have considered a dynamic design for the synchronization terms $\eta_i$. This guarantees that $\eta_i\in\mathcal{C}^1$, which is difficult to realize using static synchronization terms in view of the irregularities of the information received by each agent. In addition, the vectors $\eta_i$ are designed as in \eqref{filter}-\eqref{eta_zeros} such that the closed loop system \eqref{model_reduced} with \eqref{v_reference} and \eqref{filter}-\eqref{eta_zeros} for each agent is IOS with arbitrary IOS gains. As will be made clear in the next subsection, by employing the small gain theorem, Theorem~\ref{theorem001a}, we will show that our control objectives are achieved under some conditions that can be always satisfied. 
            \end{remark}

\subsection{Main result}
            The following theorem describes the conditions under which Objectives~\ref{objective1} and~\ref{objective2} are achieved.
        \begin{theorem}\label{theorem2}
            Consider the network of $n$ systems \eqref{model}, where the interconnection topology is described by a directed graph $\mathcal{G}=(\mathcal{N},\mathcal{E})$ and the communication process between the systems satisfies Assumption~\ref{AssumptionDelay01}.
            Suppose that each system is controlled by a control law $u_i$ satisfying Assumption~\ref{designcond}, where the corresponding reference velocity $v_{r_i}(t)$ is generated by \eqref{v_reference}, with \eqref{updatevk_leader}-\eqref{hat_v_filter} and \eqref{filter}-\eqref{epsilon}.
            Let the control gains be selected such that
                            \begin{equation}\label{small_gain_condition01}
                           	\mu_i >  1+2\cdot h^*, \quad~\mbox{for}~i\in\mathcal{N}^\sharp,
                        \end{equation}
            where $h^*>0$ is defined by \eqref{h_star}, and $$\mu_i := -\max\left({\mathcal Re} (\mu_{i,1}), {\mathcal Re} (\mu_{i,2})\right),$$
             where $\mu_{i,1}$, $\mu_{i,2}$ are the roots of $x^2+k_i^{\eta}x+\lambda_i = 0$. Then, for arbitrary initial conditions, we have
            \begin{itemize}
            \item  Objective~\ref{objective1} is achieved if $\mathcal{G}$ contains a spanning tree with a root \, $r\in{\mathcal L}$.
            \item   Objective~\ref{objective2} is achieved if $\mathcal{G}$ contains a spanning tree.
              \end{itemize}
	      \end{theorem}

\begin{proof}
        First, it should be noted from \eqref{v_reference}, \eqref{updatevk_leader}-\eqref{hat_v_filter}, and \eqref{filter}-\eqref{epsilon} that $v_{r_i}\in\mathcal{C}^1$, and $\dot{v}_{r_i}(t)= \dot{\eta}_i + \dot{\bar{v}}_{d_i}$ can be obtained from the solution of the dynamic systems \eqref{hat_v_filter} and \eqref{filter}, and is available for feedback. Therefore, applying the control law that satisfies Assumption~\ref{designcond} in \eqref{model} guarantees that the velocity tracking error $e_i(t) = v_i(t) - v_{r_i}(t)$ is uniformly bounded.
 %

%
        For each $i\in\mathcal{N}^\sharp$, let
            \begin{equation}\label{tilde_q_1}\begin{array}{cc}
              	  \tilde{p}_i := p_i - \psi_{i},  &\quad {\tilde{\psi}}_{i} := {\psi}_{i}- \frac{1}{\kappa_i}\sum_{j=1}^n a_{ij}p_{j}.\end{array}
            \end{equation}
        Using the relation $e_i = (v_i - \eta_i - \bar{v}_{d_i})$ with \eqref{filter}, one can write
        \begin{align}
                \label{dot_eta}\dot{\eta}_i=&~ -k_i^{\eta} \eta_i - \lambda_i \tilde{p}_i \\
                \label{dot_tilde_p}\dot{\tilde{p}}_i = & ~ \eta_i   + \tilde{\psi}_i + \varepsilon_i\\
                \label{dot_tilde_psi}\dot{\tilde{\psi}}_{i} = & -\tilde{\psi}_{i}- \varepsilon_i + \phi_i
        \end{align}
        for $i\in\mathcal{N}^\sharp$, with
         \begin{eqnarray}
            \label{phi}\phi_i &=& e_i - \frac{1}{\kappa_i}\sum_{j=1}^n a_{ij} \left( v_j - {\bar v}_{d_i}\right),\\
            \label{varepsilon}\varepsilon_i &=& e_i + \frac{1}{\kappa_i}\sum_{j\in\mathcal{N}_i} a_{ij}( p_{j} - p^{(i)}_{j}).
            \end{eqnarray}
         We can verify that each system \eqref{dot_eta}-\eqref{dot_tilde_psi} with output $\eta_i$ is IOS with respect to the input vectors $\phi_i$ and $\varepsilon_i$. This follows by noticing that the following estimates
        \begin{equation}\label{estimate_tildepsi}
                              	\left|
                              	{\tilde {\psi}}_{i}(t) \right|
                      	\le
                              	~e^{-(t-t_0)}\left|
                              	{\tilde {\psi}}_{i}(t_0)\right|
                          	+
                           	\sup\limits_{\varsigma\in[t_0, t]}\left| \varepsilon_i(\varsigma) \right|
                          	+
                             	\sup\limits_{\varsigma\in[t_0, t]}\left| \phi_i(\varsigma) \right|,
                	\end{equation}
            \begin{align}\label{estimate_eta}
                          	\left|\begin{array}{c}
                          	\eta_{i} (t) \\ {\tilde {p}}_{i}(t)
                          	 \end{array}\right|
                  	\le&
                          	~e^{-\mu_i(t-t_0)}\left|\begin{array}{c}
                          	\eta_i(t_0) \\ {\tilde p}_{i}(t_0)\end{array}\right|
                      	+
                          	\frac{1}{\mu_i}\sup\limits_{\varsigma\in[t_0, t]}\left| \tilde{\psi}_{i} (\varsigma) \right| \nonumber\\
                   +&
                          	~\frac{1}{\mu_i} \sup\limits_{\varsigma\in[t_0, t]}\left|  \varepsilon_{i} (\varsigma) \right|,
            	\end{align}
        hold for all $t\geq t_0$, where $\mu_i$ is defined in Theorem \ref{theorem2}. More precisely, inequality \eqref{estimate_tildepsi} indicates that system \eqref{dot_tilde_psi} is ISS with respect to the inputs $\varepsilon_i$ and $\phi_i$, with unity ISS gains. Also, \eqref{estimate_eta} implies that \eqref{dot_eta}-\eqref{dot_tilde_p} is ISS with respect to the inputs $\tilde{\psi}_i$ and $\phi_i$, with ISS gains both equal to $\frac{1}{\mu_i}$. Since the cascade connection between two ISS systems is ISS, we conclude that \eqref{dot_eta}-\eqref{dot_tilde_psi} is ISS. As a result, the system \eqref{dot_eta}-\eqref{dot_tilde_psi} with output $\eta_i$ is IOS with respect to the inputs $\varepsilon_i$ and $\phi_i$ with linear IOS gains equal to $\frac{2}{\mu}_i$ and $\frac{1}{\mu_i}$, respectively.
%

%
        %
        Therefore, all the systems \eqref{dot_eta}-\eqref{phi}, for $i\in\mathcal{N}^\sharp$, can be regarded as a system with $n^\sharp$ outputs, given by $y_i = \eta_i$, $i\in\mathcal{N}^\sharp$, and $2n^\sharp$ inputs that can be ordered as: $u_{2i}:=\varepsilon_i$, $u_{2i-1}:=\phi_i$, for $i\in\mathcal{N}^\sharp$, where $n^\sharp:= \lbv \mathcal{N}^\sharp \rbv$ denotes the number of elements in $\mathcal{N}^\sharp$. From the above analysis, we can conclude that such system is IOS, with IOS gain matrix $\Gamma^0:=\left\{ \gamma^0_{il} \right\}\in\mathbb{R}^{n^\sharp\times 2n^\sharp}$ given by
            \begin{equation*}
                \gamma^0_{il}=
                \begin{cases}
                    { 1 }/{ \mu_i} &  \text{ if} \; l=2i-1, \; i\in\mathcal{N}^\sharp, \\
                    { 2}/{ \mu_i} & \text{ if}  \; l=2i, \; i\in\mathcal{N}^\sharp, \\
                    0 & \text{ otherwise. }
                \end{cases}
            \end{equation*}
             Moreover, using \eqref{phi}-\eqref{varepsilon} with \eqref{qijversion2_1}-\eqref{eta_zeros} and the relation $v_i = (e_i + \eta_i + {\bar v}_{d_i})$ for $i\in\mathcal{N}$, one can write
             \begin{eqnarray}
             \left| u_{2i-1}(t) \right|  &\leq& \sum_{j\in\mathcal{N}_i^\sharp} \frac{a_{ij}}{\kappa_i} \left| \eta_j(t) \right| + \left| \delta_{2i-1}(t)\right|,\\
%
             \left|u_{2i}(t)\right| &\leq& \left|e_i(t)\right| + \sum_{j\in\mathcal{N}_i} \frac{a_{ij}}{\kappa_i} \int_{k_{ij}(t)T}^{t}
             \left|v_j(s)-\hat{v}_{d_j}(k_{ij}(t))\right| ds\nonumber\\
               &\leq& \sum_{j \in\mathcal{N}_i^\sharp} \frac{a_{ij}\cdot h^*}{\kappa_i}  \sup\limits_{\varsigma\in[k_{ij}(t)T, t]} \left| \eta_j (\varsigma)\right|+ |\delta_{2i}(t)|,
             \end{eqnarray}
        where we used Assumption \ref{AssumptionDelay01} and \eqref{h_star} to conclude that $(t-k_{ij}(t) T)\leq (k^* T + h):= h^*$, the set $\mathcal{N}_i^\sharp:=\{j\in \mathcal{N}^\sharp:\; (j,i)\in\mathcal{E}\}$ is used here due to \eqref{eta_zeros}, and
        \begin{align}
        \label{delta1}|\delta_{2i-1}(t)| = & \left| e_i(t)\right| +  \sum_{j\in\mathcal{N}_i} \frac{a_{ij}}{\kappa_i} \left| e_j(t) +\bar{v}_{d_j}(t) - {\bar v}_{d_i}(t) \right|,\\
        \label{delta2}|\delta_{2i}(t)| = & \sum_{j \in\mathcal{N}_i} \frac{a_{ij}\cdot h^*}{\kappa_i}  \sup\limits_{\varsigma\in[k_{ij}(t)T, t]}
        \left| \bar{v}_{d_j}(\varsigma) - \hat{v}_{d_j}(k_{ij}(t))\right|\nonumber\\
        &+ \left| e_i(t) \right| + \sum_{j \in\mathcal{N}_i} \frac{a_{ij}\cdot h^*}{\kappa_i} \sup\limits_{\varsigma\in[k_{ij}(t)T, t]}
        \left| e_j(\varsigma) \right|.
                         \end{align}

        Therefore, one can conclude that the input vectors $u_j$, $j\in\{ 1,\ldots,2n^\sharp\}$, satisfy the conditions of Theorem \ref{theorem001a}, where the elements of the interconnection matrix ${\mathcal M}:=  \left\{{\mu}_{lj}\right\}\in\mathbb{R}^{2n^\sharp\times n^\sharp}$ are obtained as
            $$
            \mu_{lj}=
            \left\{
            \begin{array}{ll}
            \frac{a_{ij}}{\kappa_i} & \quad \mbox{ if } l=2i-1, \; j \in\mathcal{N}_i^\sharp, \;  i \in\mathcal{N}^\sharp,\\
           \frac{a_{ij}\cdot h^*}{\kappa_i}  & \quad \mbox{ if } l=2i, \; j \in\mathcal{N}_i^\sharp, \;  i \in\mathcal{N}^\sharp,\\
           0 & \quad \mbox{otherwise},
            \end{array}\right.
            $$
        and where $\delta_{2i-1}(t)$ and $\delta_{2i}(t)$, for $i\in\mathcal{N}^\sharp$, satisfy \eqref{delta1}-\eqref{delta2}.  Note that in view of Assumption~\ref{designcond} and the result of Proposition~\ref{prop21}, we have $\delta_j(t)$, $j = \{1, \ldots, 2n^\sharp\}$ are uniformly bounded.
%

%
        Therefore, the elements of the closed-loop gain matrix $\Gamma:=\Gamma^0\cdot {\mathcal M}=\left\{ \bar{\gamma}_{ij} \right\}_{i, j\in\mathcal{N}^\sharp}$ in Theorem \ref{theorem001a} can be written as
              	\begin{equation}\label{bar_gamma_ij}
                    \bar{\gamma}_{ij}
                    =
              	\left\{
                  	\begin{array}{ll}
                      	\frac{ a_{ij} }{ \kappa_i\cdot \mu_i} \left(1+ 2 \cdot h^*\right), &
                      	~ \mbox{ if } \; j \in\mathcal{N}_i^\sharp, \;i \in\mathcal{N}^\sharp,\\
                      	0 & ~\mbox{ otherwise. }\end{array}\right.\nonumber
            	\end{equation}
        Taking into account the fact that $\bar{\gamma}_{ii} = 0$ and the elements of $\Gamma$ are nonnegative, one can conclude using Gersgorin disk Theorem that $\rho(\Gamma)<1$ if $\sum_{j=1}^n \bar{\gamma}_{ij}<1$, for $i\in\mathcal{N}$. Noting that $\sum_{j=1}^n \bar{\gamma}_{ij} = \sum_{j\in\mathcal{N}_i^\sharp} \frac{a_{ij}}{\mu_i\kappa_i} \left(1+ 2 \cdot h^*\right)$ and $\mathcal{N}_i^\sharp \subseteq \mathcal{N}_i$, the condition $\rho(\Gamma)<1$ is satisfied by \eqref{small_gain_condition01}.
  %

%
        Therefore, all the conditions of Theorem \ref{theorem001a} are satisfied and one can conclude that $\eta_i$, $\phi_i$ and $\varepsilon_i$, for $i\in\mathcal{N}^\sharp$, are uniformly bounded. In addition, the ISS property of \eqref{dot_eta}-\eqref{dot_tilde_psi} guarantees that $\dot{\eta}_i$, $\tilde{p}_i$ and $\tilde{\psi}_i$, for $i\in\mathcal{N}^\sharp$, are uniformly bounded. This with the result of Proposition~\ref{prop21} and \eqref{eta_zeros} lead to the conclusion that $v_{r_i}$ and $\dot{v}_{r_i}$, $i\in\mathcal{N}$, are uniformly bounded, and hence Assumption~\ref{designcond} guarantees that $e_i(t)\to 0$ as $t\to+\infty$, for $i\in\mathcal{N}$. Furthermore, using the result of Proposition~\ref{prop21} and the fact that $k_{ij}(t)\to + \infty$ as $t\to +\infty$, it can be verified from \eqref{delta1}-\eqref{delta2} that: $\delta_{j}(t)\to 0$ at $t\to +\infty$, $j\in\{1,\ldots , 2n^\sharp\}$ if $\mathcal{G}$ is rooted at $r\in\mathcal{L}\neq\emptyset$, or $\mathcal{G}$ is rooted in the case $\mathcal{L} = \emptyset$.
   %

%
        Consequently, one can conclude from Theorem~\ref{theorem001a} that $\eta_i(t)\to 0$, $\phi_i(t) \to 0$, $\varepsilon_i(t)\to 0$ for $i\in\mathcal{N}^\sharp$ as $t\to+\infty$. Since $(\eta_i(t) + e_i(t)) = (v_i(t) - {\bar v}_{d_i}(t)) \to 0$ as $t\to +\infty$, for $i\in\mathcal{N}$, the result of Proposition \ref{prop21} implies that $v_i(t) \to v_d$ as $t\to +\infty$ if $\mathcal{G}$ is rooted at $r\in\mathcal{L}\neq\emptyset$. The same proposition implies that $v_i(t) \to v_c$  as $t\to +\infty$, for some $v_c\in\mathbb{R}^m$, if $\mathcal{G}$ is rooted and $\mathcal{L} = \emptyset$.
%

%
         In addition, since system \eqref{dot_eta}-\eqref{dot_tilde_psi} is ISS, we have $\tilde{p}_i(t)\to 0$ and $\tilde{\psi}_i(t)\to 0$ as $t\to +\infty$ for $i\in\mathcal{N}^\sharp$. Using \eqref{tilde_q_1}, one gets $\kappa_i(\tilde{p}_i + \tilde{\psi}_i) = \sum_{j=1}^n a_{ij} (p_i -p_j)$ is uniformly bounded and $\sum_{j=1}^n a_{ij} (p_i(t) - p_j(t)) \to 0$ as $t\to+\infty$ for all $i\in\mathcal{N}^\sharp$. This with the fact that $\kappa_i = 0$ for $i\in\mathcal{N}\setminus \mathcal{N}^\sharp$ lead to the conclusion that $(\mathbf{L}\otimes \mathbf{I}_m)p(t) \to 0$ as $t\to+\infty$, where $\mathbf{L}$ is the Laplacian matrix of the interconnection graph $\mathcal{G}_w$, $\mathbf{I}_m$ is the $m\times m$ identity matrix, $p\in\mathbb{R}^{nm}$ is the vector containing all $p_i$ for $i\in\mathcal{N}$, and $\otimes$ is the Kronecker product. Finally, since $(\mathbf{L}\otimes \mathbf{I}_m)p = 0$ implies that $p_1 = \ldots = p_n$ if $\mathcal{G}_w$ contains a spanning tree \cite{Ren:Beard:2005:ieeetac}, we conclude that  $(p_i(t) - p_j(t)) \to 0$ as $t\to + \infty$, for all $i, j \in\mathcal{N}$ if $\mathcal{G}$ is rooted. The proof is complete.
\end{proof}
        Theorem \ref{theorem2} gives a solution to the synchronization problem of the class of nonlinear systems \eqref{model} with relaxed communication requirements. In fact, each agent needs to send its information to its prescribed neighbors only at some instants of time.
        This information transfer is also subject to constraints inherent to the communication channels such as irregular communication delays and packet loss. An important feature of the above result is that it gives sufficient conditions for synchronization, given in \eqref{small_gain_condition01}, that are topology-independent and can be easily satisfied with an appropriate choice of the control gains. Notice that the constant $h^*:=(k^*T+ h)$ can be easily estimated in practice, and is simply defined as the maximum blackout interval of time an individual agent does not receive information from each one of its neighbors. Then, the control gains, namely $k_i^\eta$ and $\lambda_i$, can be freely selected to satisfy \eqref{small_gain_condition01}; in particular, the variable $\mu_i$ can be made arbitrarily large, which is advantageous in the case where the parameter $h^*$ is roughly or over estimated. On the other hand, condition \eqref{small_gain_condition01} is equivalent to $0< h^* < \frac{1}{2}(\mu_i - 1)$, which specifies the maximal allowable time interval during which each agent can run its control algorithm without receiving new information from its neighbors. This allowable interval of time does not rely on some centralized information on the interconnection topology between the systems and can be made arbitrarily large. It should be also pointed out that the results of Theorem~\ref{theorem2} are obtained under mild assumptions on the interconnection graph $\mathcal{G}$. In this regard, note that condition \eqref{small_gain_condition01} is imposed for all agents $i\in\mathcal{N}^\sharp$, where, in view of the assumptions on $\mathcal{G}$, the set $\mathcal{N}\setminus\mathcal{N}^\sharp$ contains at most one element.

        \begin{remark}
                The term $\epsilon_{ij}(t)$ in \eqref{qijversion2_1} can be selected in different ways using the estimates of the desired velocity of the $i$-th agent. The choices other than \eqref{epsilon} include:
                \begin{equation}
                \epsilon_{ij}(t):= \int_{k_{ij}(t) T}^{t} \hat{v}_{d_i}(\lfloor s/T\rfloor)ds,
                                \label{epsilonA}\end{equation}
                 \begin{equation}\epsilon_{ij}(t):= \int_{k_{ij}(t) T}^{t} {\bar v}_{d_i}(s)ds,\label{epsilonB}\end{equation}
                  and
                 \begin{equation}\epsilon_{ij}(t):= {\bar v}_{d_i}(k_{ij}(t) T)\cdot (t-k_{ij}(t) T). \label{epsilonC}\end{equation}
                  In view of Proposition~\ref{prop21}, any of the choices \eqref{epsilon}, \eqref{epsilonA}-\eqref{epsilonC} can be used for our purposes.
            \end{remark}
%
%
        The control scheme in Theorem \ref{theorem2} can be applied to the case where the desired velocity is available to all systems, {\it i.e.,} $\mathcal{L}=\mathcal{N}$. In this case, the observer \eqref{updatevk_leader}-\eqref{updatevk_f2} is not needed and the following result, which can be shown following similar arguments as the proof of Theorem \ref{theorem2}, is valid.
        \begin{corollary}\label{cor1}
        Consider the network of $n$-systems \eqref{model}, where the interconnection topology is described by a directed graph $\mathcal{G}$ and the communication process between the systems satisfies Assumption~\ref{AssumptionDelay01}. Suppose that each system is controlled by a control law $u_i$ satisfying Assumption~\ref{designcond}, where the corresponding reference velocity $v_{r_i}(t)$ is defined in \eqref{v_reference}, where ${\bar v}_{d_i} \equiv v_d$, $i\in\mathcal{N}$, and $\eta_i(t)$ is obtained from \eqref{filter}-\eqref{eta_zeros} with $p^{(i)}_{j}(t):=p_j( k_{ij}(t) T)+v_d \cdot (t-k_{ij}(t) T)$, $(j,i)\in\mathcal{E}$, and $k_{ij}(t)$ is given in \eqref{k_t}. Let the control gains satisfy \eqref{small_gain_condition01}. Then, $v_i(t)\to v_d$ and $(p_i(t)-p_j(t)) \to 0$ as $t\to +\infty$ for all $i,j\in\mathcal{N}$ and for arbitrary initial conditions if $\mathcal{G}$ contains a spanning tree.
        \end{corollary}

\section{Application to Euler-Lagrange systems}\label{app_example}
        In this section, we apply the proposed approach to the class of fully-actuated heterogeneous Euler-Lagrange systems. The systems dynamics are given by
        \vspace{-0.1 in}
            \begin{equation}\label{model_case1}\begin{array}{lcl}
            \dot{p}_i &=& v_i\\
                \dot{v}_i &=& M_i(p_i)^{-1}\left(u_i-C_i(p_i,v_i)v_i-G_i(p_i)\right),\end{array}\vspace{-0.1 in}
            \end{equation}
        for $i\in\mathcal{N}$, where $p_i\in\mathbb{R}^m$ is the vector of generalized configuration coordinates, $u_i\in\mathbb{R}^m$ is the vector of torques associated with the $i^{th}$ system, $M_i(p_i)$, $C_i(p_i,v_i)v_i$, and $G_i(p_i)$ are the inertia matrix, the vector of centrifugal/coriolis forces, and the vector of potential forces, respectively. The inertia matrices $M_i(p_i)$ are symmetrical and positive definite uniformly with respect to $p_i$. Other common properties of Euler-Lagrange systems (\ref{model_case1}) are as follows:
%

%
         \begin{itemize}
                    \item[P.1] The matrix $\dot{M}_i(p_i)-2C_i(p_i,v_i)$ is skew symmetric.
                    \item[P.2] There exists $k_{c_i}\ge 0$ such that $|C_i(p_i, x) y|\leq k_{c_i}|x|\cdot|y|$ holds for all $p _i,~x, ~y\in\mathbb{R}^m$. In addition, $M_i(p_i)$ and $G_i(p_i)$ are bounded uniformly with respect to  $p_i$.
                    \item[P.3] Each system in \eqref{model_case1} admits a linear parametrization of the form  $M_i(p_i)\dot{x}_{i}+C_i(p_i,v_i)x_i
                        +G_i(p_i)=Y_i(p_i, v_i, x_i, \dot{x}_i)\theta_i$,
                        where $Y_i(p_i, v_i, x_i, \dot{x}_i)$ is a known regressor matrix and $\theta_i\in\mathbb{R}^{k}$ is the constant vector of the system's parameters.
        \end{itemize}
        We assume that the systems are subject to model uncertainties; the  parameters $\theta_i$ in P.3~are unknown. We aim to achieve Objectives \ref{objective1} and \ref{objective2} for the Euler-Lagrange systems \eqref{model_case1} under a directed interconnection graph $\mathcal{G}$ and the communication constraints described in Section~\ref{section_comm}.
 %

%
        For this purpose, we consider the following control input in \eqref{model_case1}
            \begin{eqnarray}\label{control_EL}
            u_i &=& Y_i(p_i, v_i, v_{r_i}, \dot{v}_{r_i}) \hat{\theta}_i - k_i^e e_i,\\
            \label{adapt_EL}\dot{\hat{\theta}}_i &=& \Pi_i Y_i (p_i, v_i, v_{r_i}, \dot{v}_{r_i})^{\top} e_i,
            \end{eqnarray}
        for $i\in\mathcal{N}$, where the matrix $\Pi_i$ is symmetric positive definite, $Y_i$ is defined in P.3, $\hat{\theta}_i\in\mathbb{R}^k$ is an estimate of the parameters, $k_i^e>0$, and $e_i = (v_i - v_{r_i})$  with
         \begin{eqnarray}
         &&\begin{array}{ccl}v_{r_i} &=& \eta_i + {\bar v}_{d_i},\end{array}\\
                &&\left\{\begin{array}{ccl}
                            \dot{\eta}_i &=& - k_i^{\eta} \eta_i - \lambda_i (p_i - \psi_i) \\
                            \dot{\psi}_i &=& - \psi_i + {\bar v}_{d_i} + \frac{1}{\kappa_i}\sum_{j=1}^n a_{ij} p^{(i)}_{j}(t)
                            \end{array}\right., ~~i\in\mathcal{N}^\sharp,\\
                 &&\begin{array}{ccl}
                            \dot{\eta}_i &=& \eta_i \equiv 0,
                            \end{array} \qquad i\in\mathcal{N}\setminus\mathcal{N}^\sharp,
                \end{eqnarray}
            where $p^{(i)}_{j}(t):=p_j( k_{ij}(t) T)+ \hat{v}_{d_j} (k_{ij}(t))\cdot (t-k_{ij}(t) T)$, the control gains are defined as in Theorem~\ref{theorem2}, $k_{ij}(t)$ is defined in \eqref{k_t}, $\bar{v}_{d_i}$ is obtained from \eqref{hat_v_filter_leader}-\eqref{hat_v_filter} with the discrete-time observer
              \begin{eqnarray}
              \hat{v}_{d_i}(\sigma)&\equiv&  v_d, \qquad i\in {\mathcal L},\\
               \label{observer_ELsystems} \hat{v}_{d_i}(\sigma+1)&=&
                    \frac{1}{\left\bracevert{N_i(\sigma)}\right\bracevert}\sum\limits_{j\in N_i(\sigma)}\hat{v}_{d_{ij}}(\sigma) \quad ~i\in {\mathcal F},
               \end{eqnarray}
        where $\hat{v}_{d_{ij}}(\sigma)$ is given in \eqref{updatevk_f2} and $N_i(\sigma)$ is defined after \eqref{updatevk_f2}.
        Then, the following result is valid.
         \begin{corollary}\label{cor_EL}
         Consider the network of $n$ Euler-Lagrange systems \eqref{model_case1} interconnected according to $\mathcal{G}$ and suppose that Assumption~\ref{AssumptionDelay01} holds. For each system, let the control input be given in \eqref{control_EL} with \eqref{adapt_EL}-\eqref{observer_ELsystems}, and suppose condition \eqref{small_gain_condition01} is satisfied. Then, Objective 1 and Objective 2 are achieved under the conditions on the interconnection graph $\mathcal{G}$ given in Theorem~\ref{theorem2}.
         \end{corollary}

         The proof of this result follows from Theorem~\ref{theorem2} by noting that $v_{r_i}\in\mathcal{C}^1$ and $\dot{v}_{r_i}$ is well defined, and the control law \eqref{control_EL}-\eqref{adapt_EL} is the standard adaptive control scheme proposed in \cite{slotine1987adaptive} for Euler-Lagrange systems that satisfies Assumption~\ref{designcond}. In fact, using the Lyapunov function $V_i= \frac{1}{2}(e_i^{\top}M_i(p_i)e_i + \tilde{\theta}_i \Pi^{-1} \tilde{\theta}_i)$, with $\tilde{\theta}_i = (\hat{\theta}_i - \theta_i)$, one can show that $e_i\in\mathcal{L}_2 \cap \mathcal{L}_{\infty}$ and $\hat{\theta}_i \in\mathcal{L}_{\infty}$, leading to the first point in Assumption~\ref{designcond}. Also, properties P.2 and P.3 guarantee that $\dot{e}_i\in\mathcal{L}_{\infty}$ if  $v_{r_i},~\dot{v}_{r_i}\in\mathcal{L}_{\infty}$. Then, invoking Barb\u{a}lat Lemma, one can conclude that if $v_{r_i},~\dot{v}_{r_i}$ are uniformly bounded, then $e_i(t)\to 0$ as $t\to +\infty$, which is the second point in Assumption~\ref{designcond}.

        \begin{remark}
        The control scheme in Corollary~\ref{cor_EL} extends the relevant literature dealing with Euler-Lagrange systems with communication constraints \cite[for instance]{chopra2006passivity, chung:2009, Nuno11, Nuno13, Wang:2013, Abdessameud:Polushin:Tayebi:2013:ieeetac} to the case where a non-zero final velocity is assigned to the team, the communication between agents is intermittent and subject to varying delays and possible packet loss, and under a directed interconnection graph that contains a spanning tree. In addition, this control scheme extends the work in \cite{Wang:flocking:2013} to the case of intermittent and delayed communication without using a centralized information on the interconnection topology. Note that in \cite{Wang:flocking:2013}, Objective~\ref{objective2} is achieved, in the case of delay-free continuous-time communication between agents, under some topology-dependent conditions.
        \end{remark}

\section{ Simulation Results }
%
%
         We provide in this section simulation results for the example in Section~\ref{app_example}. Specifically, we consider a network of ten Lagrangian systems; $\mathcal{N}=\{1,\ldots,10\}$, modeled by equations \eqref{model_case1} with $m=2$, $p_i := (p_{i_1}, p_{i_2})^{\top}$, $v_i:=( v_{i_1}, v_{i_2})^{\top}$, and
         \begin{equation}
         M_i(p_i) = \left(\begin{array}{cc}
         \theta_1 + 2 \theta_2\cos( p_{i_2}) & ~\theta_3+\theta_2\cos( p_{i_2})\\
         \theta_3+\theta_2\cos( p_{i_2}) & ~\theta_3\end{array}\right),\nonumber
         \end{equation}
         \begin{equation}
         C_{i}(p_i, v_i) = \left(\begin{array}{cc}
         -\theta_2\sin( p_{i_2}) v_{i_2} & ~-\theta_2\sin( p_{i_2})( v_{i_1}+ v_{i_2})\\
         \theta_2\sin( p_{i_2}) v_{i_1} & ~0\end{array}\right),\nonumber
         \end{equation}
         \begin{equation}
         G_i(p_i) = \left(\begin{array}{c} g\theta_5\cos( p_{i_1})+g\theta_4\cos( p_{i_1}+ p_{i_2})\\ g\theta_4\cos( p_{i_1}+ p_{i_2})\end{array}\right),\nonumber
         \end{equation}
         where $g=9.81~ \mathrm{m/sec^{2}}$, the variables $\theta_k$, $k=1,\ldots,5$, are given as: $\theta_1=(m_1l^2_{c1} + m_2(l^2_1 + l^2_{c2})+I_1+I_2)$, $\theta_2=m_2l_1l_{c2}$, $\theta_3=(m_2l_{c2}^2+I_2)$,  $\theta_4=m_2l_{c_2}$, and $\theta_5=(m_1l_{c1}+m_2l_1)$, with $m_1=m_2=1~ \mathrm{kg}$, $l_1=l_2=0.5~\mathrm{m}$, $l_{c1}=l_{c2}=0.25~ \mathrm{m}$, and $I_1=I_2=0.1~ \mathrm{kg/m^2}$.
         The parametrization satisfying property P.3 for each system is given as:
                  $Y_i(p_i, v_i, x_i, \dot{x}_i)= [Y_{i_{jk}}]\in\mathbb{R}^{2\times 5}$, with $Y_{i_{11}}= \dot{x}_{i_1}$, $Y_{i_{12}}= \cos( p_{i_2})(2 \dot{x}_{i_1}+ \dot{x}_{i_2})-\sin( p_{i_2})( x_{i_1}  v_{i_2}+  x_{i_2}( v_{i_1}+ v_{i_2}))$, $Y_{i_{13}}=  \dot{x}_{i_2}$, $Y_{i_{14}}= Y_{i_{24}}= g \cos( ^ip_i+ p_{i_2})$, $Y_{i_{15}}= g \cos( ^ip_i)$, $Y_{i_{21}}=Y_{i_{25}}=0$, $Y_{i{22}}= \dot{x}_{i_1}\cos( p_{i_2}) +  x_{i_1}  v_{i_1}\sin( p_{i_2})$, $Y_{i_{23}}= \dot{x}_{i_1}+ \dot{x}_{i_2}$, for any $x_i:=( x_{i_1}, x_{i_2})^{\top}$. The vector of estimated parameters is $\hat{\theta}=(\hat{\theta}_1, \hat{\theta}_2, \hat{\theta}_3, \hat{\theta}_4, \hat{\theta}_5)^{\top}$.
%

 %
         The systems in the network are interconnected according to the directed graph $\mathcal{G}$ given in Fig. \ref{fig:graph},
         with the node labeled $4$ being one of its roots; $r=4$. For the communication process, we set the sampling period $T = 0.1 \sec$, which means that for each $(j,i)\in\mathcal{E}$, agent $j$ can send its information to agent $i$ only at instants $kT$, $k\in\mathbb{Z}_+$. The delays and packet dropout are generated for each communication link as follows.
        For each $(j,i)\in\mathcal{E}$, and at each instant $k T$, we pick the information $\mbf{p}_j^i(\bar{k})$, where $\bar{k}\in\mathbb{Z}$ is randomly selected in the interval $[k-10, k]$. This information is then delayed by $\tau \in [0.15, 0.25]~\mathrm{sec}$ and considered as received by agent $i$. Due to the random choice of $\bar{k}$, a simple logic is implemented to avoid sending information at a future $k$ with the same $\bar{k}$. This way, the parameter $h^*$ is estimated to be $1.3~\sec$, the variable $k_{ij}(t)$, for all $t>0$, and the set $\mathcal{N}_i(\sigma)$ can be easily obtained, the information of agent $j$ at some instants $kT$ are lost (not submitted), and the information received by agent $i$ is randomly delayed. The intermittent nature of the communication process as well as varying communication delays and packet dropout are illustrated in Fig.~\ref{fig:sine:test}, which shows the received discrete-time signal $\tilde{\phi}(k)$ when the signal $\phi(t) = 2\sin(t)$ is sent according to the communication process described above.
\begin{figure}[b]
\centering
\begin{minipage}[h]{0.48\columnwidth}
\centering
\includegraphics[scale = 0.08]{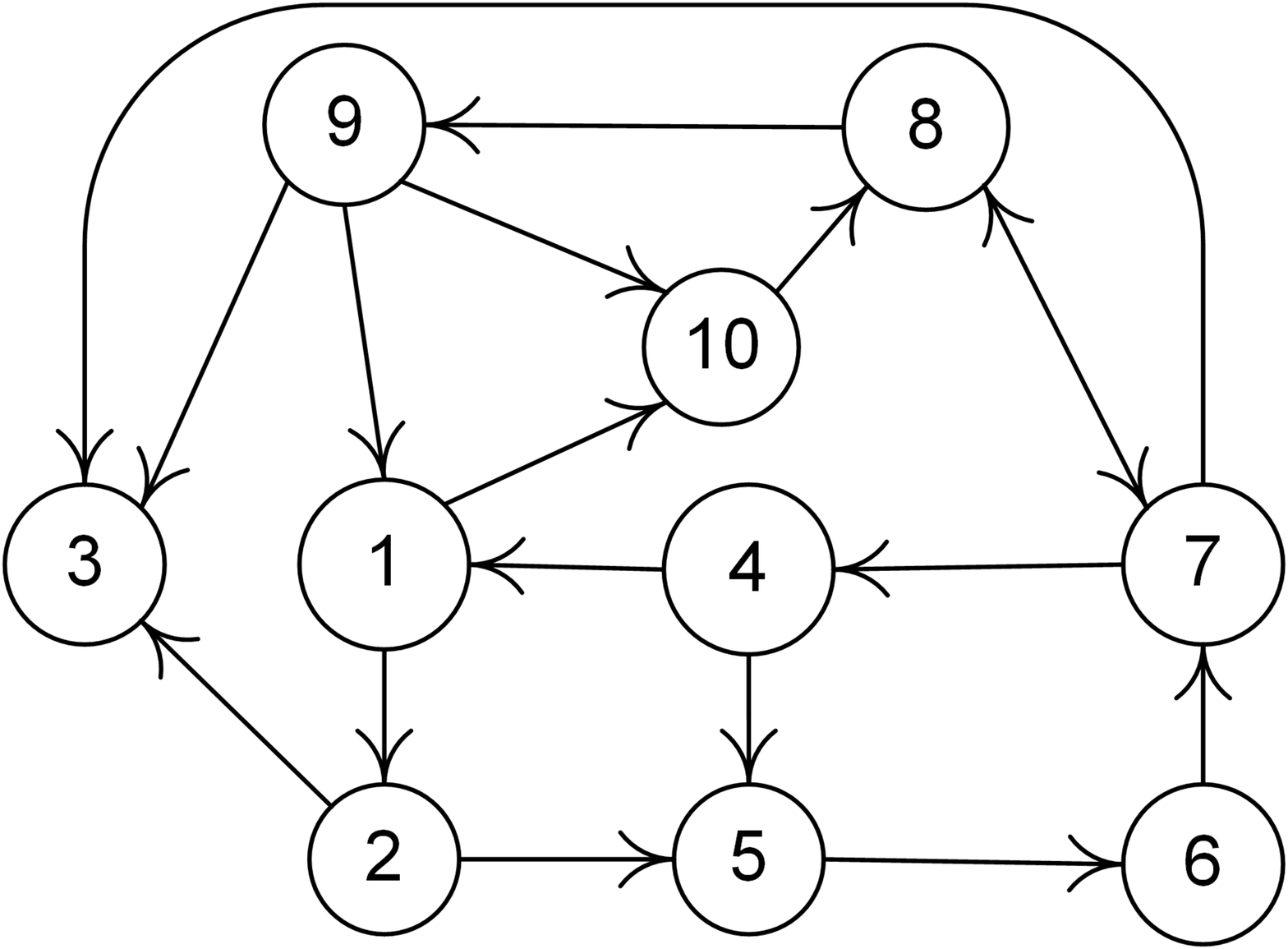}
\vspace{-0.1 in}\caption{\footnotesize interconnection graph  $\mathcal{G}$.} \label{fig:graph}
\end{minipage}
 \hfill
\begin{minipage}[h]{0.48\columnwidth}
\centering
\includegraphics[width=1\columnwidth, height = 3.2 cm]{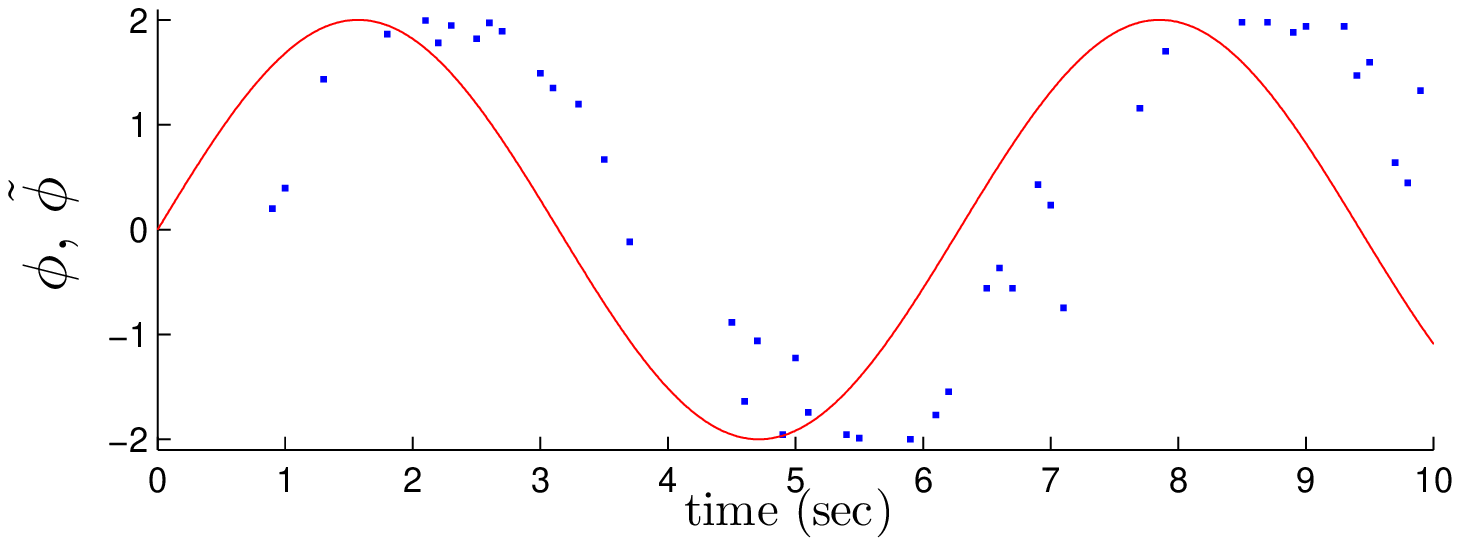}
\vspace{-0.15in}
\caption{\footnotesize Input and output of a communication channel, $\phi(t) = 2\sin(t)$.} \label{fig:sine:test}
\end{minipage}
\end{figure}

        We implement the control scheme developed for Euler-Lagrange systems in Section~\ref{app_example}. First, we consider the case where $\mathcal{L}=\{ 1, 4\}$, which indicates that the systems labeled $1$ and $4$ are the only systems having access to the desired velocity given by $v_d = (0.3, 0.6)^{\top}~\mathrm{rad/sec}$. The observer \eqref{updatevk_leader}-\eqref{updatevk_f2} is updated at $T$, and the control gains are selected as: $k_i^p = k_i^d = 2$, $\Pi_i = 0.3 \mathbf{I}_5$, $k_i^e = 10$, $\lambda_i = 13$,  $k_i^{\eta} = 2\sqrt{\lambda_i}$. Note that this choice of the gains satisfies condition \eqref{small_gain_condition01} with $\mu_i = \sqrt{\lambda_i}$. The weights of the communication links of $\mathcal{G}_w$, which is the same as $\mathcal{G}$ with assigned weights on its links, are set such that $\kappa_i = 1$.

        Fig.~\ref{fig:positions:leaders} and Fig.~\ref{fig:velocities:leaders} illustrate the relative positions and relative velocities defined as $p_{1i}= (p_{1i_1}, p_{1i_2})^{\top}:= (p_1 - p_i)$, for $i=2, \ldots, n$, and $v_{1j}= (v_{1j_1}, v_{1j_2})^{\top}:= (v_1 - v_j)$ for $j=d,~2, \ldots,~n$, where subscript `$d$' is used for the desired velocity. It is clear that all agents synchronize their positions and velocities with the desired velocity. The output of the discrete-time observer is given in Fig. \ref{fig:velocity:estimates:leaders}, with $\hat{v}_{d_i} = (\hat{v}_{d_{i_1}}, \hat{v}_{d_{i_2}})^{\top}$, where it can be seen that the desired velocity estimate of each agent converges to the desired velocity available to the leader agents.

\begin{figure}[h]
\centering
\includegraphics[width=.7\columnwidth]{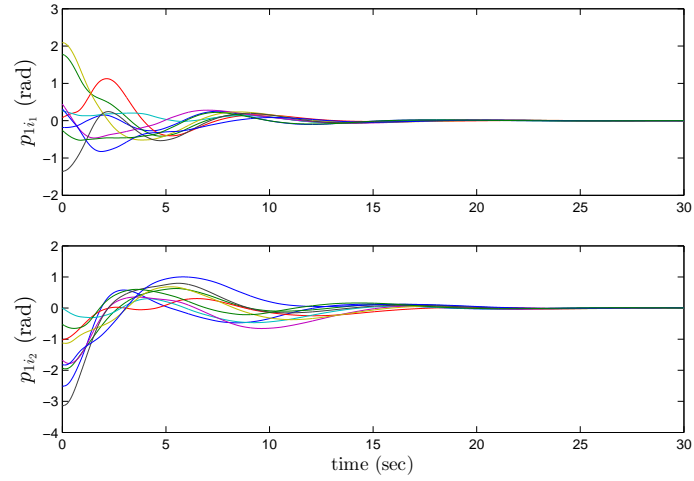}
\vspace{-0.15in}
\caption{\footnotesize Relative position vectors in the case of $\mathcal{L}=\{1, 4\}$.} \label{fig:positions:leaders}
\end{figure}

\begin{figure}[h]
\centering
\includegraphics[width=.7\columnwidth]{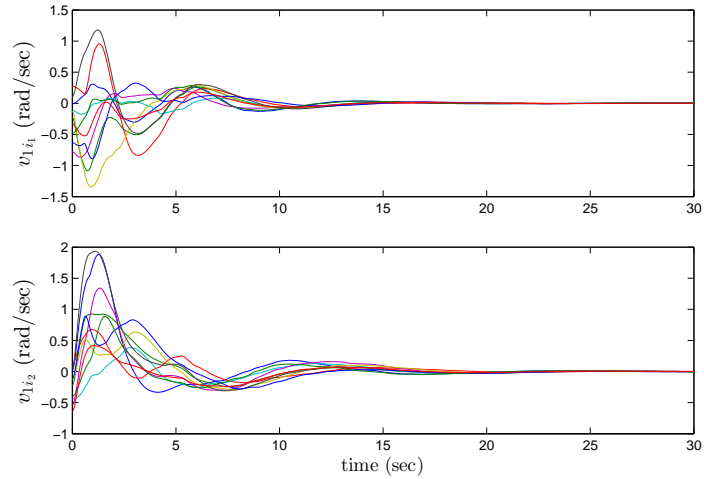}
\vspace{-0.15in}
\caption{\footnotesize Relative velocity vectors in the case of $\mathcal{L}=\{1, 4\}$.} \label{fig:velocities:leaders}
\end{figure}

\begin{figure}[h]
\centering
\includegraphics[width=.7\columnwidth]{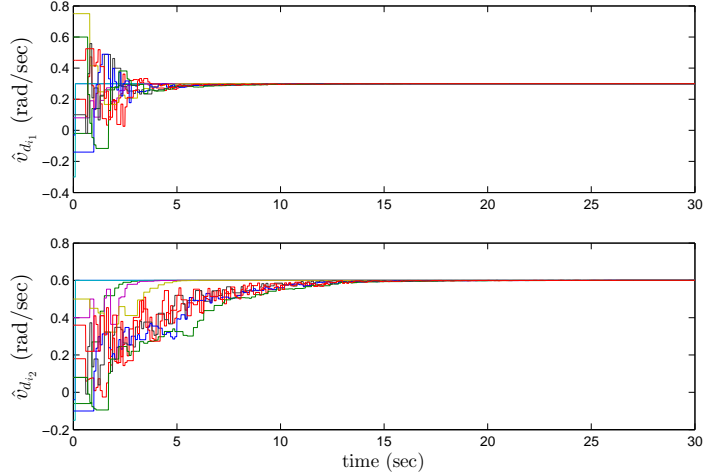}
\vspace{-0.15in}
\caption{\footnotesize Desired velocity estimates in the case of $\mathcal{L}=\{1, 4\}$.} \label{fig:velocity:estimates:leaders}
\end{figure}

Next, we consider the case where $\mathcal{L} = \emptyset$. Using the same above control parameters, the obtained results are shown in Fig.~\ref{fig:positions:leaderless}-\ref{fig:velocity:estimates:leaderless} where $v_{1j}$ is defined for $j=2,\ldots, n$. These figures show that all systems synchronize their positions, and their velocities converge to the final velocity dictated by the output of the discrete-time velocity estimator.

\begin{figure}[h]
\centering
\includegraphics[width=.7\columnwidth]{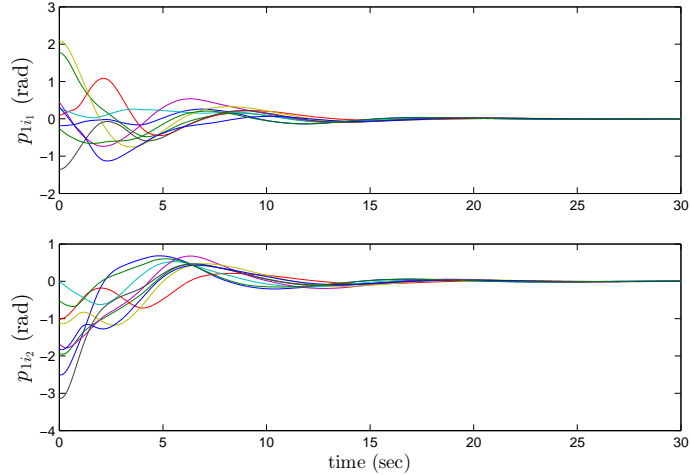}
\vspace{-0.15in}
\caption{\footnotesize Relative position vectors in the case of $\mathcal{L}  = \emptyset$.} \label{fig:positions:leaderless}
\end{figure}

\begin{figure}[h]
\centering
\includegraphics[width=.7\columnwidth]{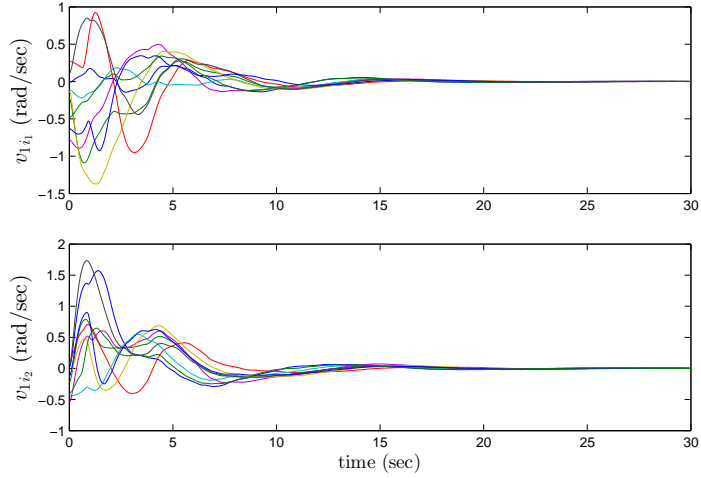}
\vspace{-0.15in}
\caption{\footnotesize Relative velocity vectors  in the case of $\mathcal{L}  = \emptyset$.} \label{fig:velocities:leaderless}
\end{figure}

\begin{figure}[h]
\centering
\includegraphics[width=.7\columnwidth]{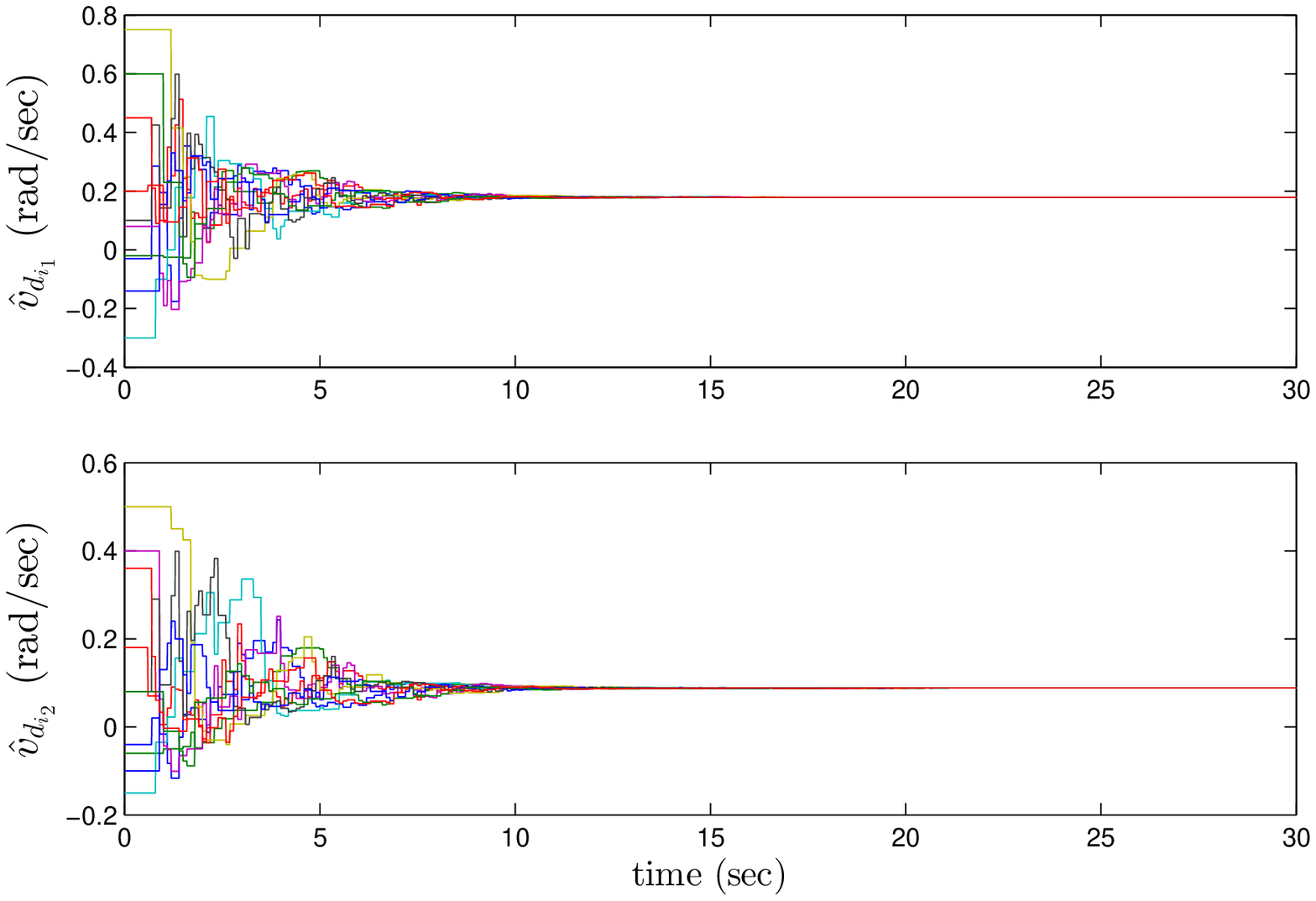}
\vspace{-0.15in}
\caption{\footnotesize Desired velocity estimates  in the case of $\mathcal{L}  = \emptyset$.} \label{fig:velocity:estimates:leaderless}
\end{figure}

\section{Conclusion}
\vspace{-0.1 in}
        We addressed the synchronization problem of second-order nonlinear multi-agent systems interconnected under directed graphs. Using the  small-gain framework, we proposed a distributed control algorithm that achieves position synchronization in the presence of communication constraints. In contrast to the available relevant literature, the proposed approach guarantees that all agents velocities match a desired velocity available to only some leaders (or a final velocity agreed upon by all agents in the leaderless case), while information exchange between neighboring agents is allowed at irregular discrete time-intervals in the presence of irregular time-delays and possible packet loss. In fact, we proved that synchronization is still achieved even if each agent in the team runs its control algorithm without receiving any information from its neighbors during some allowable intervals of time. The conditions for synchronization derived in this paper can be satisfied by an appropriate choice of the control gains.
        Future research will consider the extension of this work to the case of variable desired velocity.

\appendix

\section{Proof of Proposition~\ref{prop21}}\label{proof_prop_observer}
\vspace{-0.1in}
        Consider the consensus algorithm \eqref{updatevk_leader}-\eqref{updatevk_f2}. The interaction between the agents in the system \eqref{updatevk_leader}-\eqref{updatevk_f2} is described by a directed graph $\mathcal{G}_s = (\mathcal{N}, \mathcal{E}_s)$, which can formally be obtained from the graph $\mathcal{G}$ by modifying some of the its links, as follows: (1) removing the incoming arcs to each leader node (or agent), (2) adding a directed link from any leader node to any other leader node, and (3) adding a self arc to each node in the graph. It is straightforward to verify that, if the directed graph $\mathcal{G}$ is rooted at $r\in\mathcal{L}$, then $\mathcal{G}_s$ is also rooted at $r\in\mathcal{L}$. In the case of no leaders (${\mathcal L}=\emptyset$), the above modifications reduce to adding a self arc to each node; in this case, it is trivial that $\mathcal{G}_s$ is rooted if $\mathcal{G}$ is rooted.

In view of the above discussion, the consensus algorithm \eqref{updatevk_leader}-\eqref{updatevk_f2} can be formally written as
            \begin{equation}
                     \hat{v}_{d_i}(\sigma+1)=\frac{1}{\left\bracevert{\bar{N}_i(\sigma)}\right\bracevert}\sum\limits_{j\in \bar{N}_i(\sigma)}\hat{v}_{d_j}(\sigma-\hat\tau^{(j,i)}(\sigma)),
                                    \label{updatevk02}
              \end{equation}
        for all $i\in {\mathcal N}$, where\vspace{-0.2in}
                \begin{equation}\label{N_inproof}\bar{N}_i(\sigma)=\left\{ \begin{array}{ll}
                {N}_i(\sigma)& \quad \mbox{ for}~ i\in{\mathcal F},\\ {\mathcal L}& \quad \mbox{ for}~ i\in {\mathcal L},\end{array}\right.
                \end{equation}
        and $\hat\tau^{(j,i)}(\sigma)$ is a delay that takes some integer value at $\sigma T$ and, in view of Assumption \ref{AssumptionDelay01} and \eqref{h_star}, satisfies $\hat\tau^{(j,i)}(\sigma)\le h_\sigma^*:=\left\lceil h^*/T\right\rceil$ for all $\sigma = 0, 1, \ldots$. Note that $\hat\tau^{(i,i)}(\sigma) = 0$, and $\hat\tau^{(j,i)}(\sigma) = 0$ for all $\sigma$ if $i,j\in\mathcal{L}$.

        Let $\mathcal{G}_s(\sigma)=(\mathcal{N}, \mathcal{E}_s(\sigma))$, where $(j,i)\in \mathcal{E}_s(\sigma)$ only if $j \in \bar{N}_i(\sigma)$, which defines the set of those agents whose information is used in the update rule of agent $i$ at instants $\sigma T$.
        It is clear that $\mathcal{G}_s(\sigma)$ is a directed graph with at most one directed link connecting each ordered pair of distinct nodes and with exactly one self arc at each node.

        According to Theorem 2 in \cite{Cao:etal:2008:2:SIAMJCO}, the states of \eqref{updatevk02} satisfy $\hat{v}_{d_i}(\sigma)\to v_c$ exponentially as $\sigma\to +\infty$, $i\in\mathcal{N}$, for some $v_c \in\mathbb{R}^m$, if the sequence of graphs $\mathcal{G}_s(0),\mathcal{G}_s(1), \ldots $ is repeatedly jointly rooted. We claim that the latter condition is satisfied under the assumptions of Proposition~\ref{prop21}. To show this, pick an arbitrary $\rho \in \mathbb{Z}_+$ and consider the composition of graphs $\bar{\mathcal{G}}_s(\rho) = \mathcal{G}_s(\rho+h^*_\sigma)\circ {\mathcal G}_s(\rho+h^*_\sigma-1)\circ\ldots\circ{\mathcal G}_s(\rho)$. Since $\mathcal{G}_s(\rho)$ contains self arcs on each node, for all $\rho$, the edges of $\mathcal{G}_s(\rho+h^*_\sigma)$, ${\mathcal G}_s(\rho+h^*_\sigma-1)$, $\ldots, {\mathcal G}_s(\rho)$ are also edges in $\bar{\mathcal{G}}_s(\rho)$.

        Consider first the case where $\mathcal{L} \neq \emptyset$ and $\mathcal{G}$ is rooted at $r\in\mathcal{L}$. Equation \eqref{N_inproof} implies that ${\mathcal G}_s(\rho)$ contains a directed link from any leader node to any other leader node, for all $\rho$. In addition,
        in view of the definition of $\mathcal{G}_s=(\mathcal{N}, \mathcal{E}_s)$, Assumption \ref{AssumptionDelay01} implies that
        for each $i\in\mathcal{F}$ and $(j,i)\in \mathcal{E}_s$, $j\neq i$, the information $\hat{v}_{d_j}$ is successfully delivered to agent $i$ at least once per $h^*_\sigma:=\left\lceil h^*/T\right\rceil$ sampling periods. Therefore, it can be verified that, for each $i\in\mathcal{F}$, if $(j,i)$ is an edge in $\mathcal{G}_s$, then $(j,i)$ is also an edge in the composition of graphs $\bar{\mathcal{G}}_s(\rho)$. In fact, if $(j,i)\in\mathcal{E}_s$, the definition of $N_i(\sigma)$ implies that $(j,i)$ is an edge in at least one of the graphs $\mathcal{G}_s(\rho)$, ${\mathcal G}_s(\rho+1)$, $\ldots$, ${\mathcal G}_s(\rho+h^*_\sigma)$. Consequently, one can conclude that all the edges of $\mathcal{G}_s$ are also edges in the composition of graphs $\bar{\mathcal{G}}_s(\rho)$, and therefore, $\bar{\mathcal{G}}_s(\rho)$ is rooted at $r\in\mathcal{L}$  since $\mathcal{G}_s$ is rooted at $r\in\mathcal{L}$. As a result, the sequence of graphs $\mathcal{G}_s(0),\mathcal{G}_s(1), \ldots ~$ is repeatedly jointly rooted. Similar arguments can be used to show that $\bar{\mathcal{G}}_s(\rho)$ is rooted in the case where $\mathcal{L} = \emptyset$ and $\mathcal{G}$ contains a spanning tree, and hence rooted.

        Now, since the states of the leaders are fixed, {\it i.e.,} $\hat{v}_{d_i}(\sigma)\equiv v_d$ for all $\sigma$ and $i\in\mathcal{L}$, one can  conclude that $v_c = v_d$ in the case where $\mathcal{L}\neq \emptyset$. The rest of the proof follows in view of the dynamics \eqref{hat_v_filter}, for $i\in\mathcal{F}$, which can be rewritten as $\ddot{\tilde{\epsilon}}_i = -k_i^d \dot{\tilde{\epsilon}}_i - k_i^p \tilde{\epsilon}_i + k_i^p (\hat{v}_{d_i}(\lfloor t/T \rfloor) - v_c)$, $\tilde{\epsilon}_i := {\bar v}_{d_i} - v_c$, and describe the dynamics of an asymptotically stable system with an exponentially convergent perturbation term.


\end{document}